\documentclass[preprint,3p]{elsarticle}

\usepackage{amssymb}
\usepackage{amsmath}
\usepackage{amsthm}
\usepackage{epsfig}
\usepackage{tikz}
\usepackage{bbm}
\usepackage{amssymb}
\usepackage{amscd}
\usepackage{lineno}
\usepackage{float}
\usepackage[caption=false]{subfig}
\usepackage{algorithm}
\usepackage{algpseudocode}
\usepackage{hyphenat}
\usepackage{booktabs}
\usepackage[nowarn]{glossaries}
\usepackage[hidelinks]{hyperref}
\usepackage{tabularx}
\usepackage{psfrag}

\newcommand\VRule[1][\arrayrulewidth]{\vrule width #1}
\newcommand{\myps}{
\psfrag{a}{$a$}
\psfrag{b}{$b$}
\psfrag{N(a)}{$N(a)$}
\psfrag{N(b)}{$N(b)$}
\psfrag{Nb}{$N_b$}
\psfrag{Na}{$N_a$}
\psfrag{N2a}{$N^2_a$}
\psfrag{N2b}{$N^2_b$}
\psfrag{N2(b)}{$N^2(b)$}
\psfrag{N2(a)}{$N^2(a)$}
\psfrag{U}{$U$}
\psfrag{V}{$V$}
\psfrag{P}{$P$}
\psfrag{1}{1}
\psfrag{2}{2}
\psfrag{3}{3}
\psfrag{4}{4}
\psfrag{5}{5}
\psfrag{6}{6}
\psfrag{7}{7}
\psfrag{8}{8}
\psfrag{9}{9}
\psfrag{C}{$C$}
\psfrag{C'}{$C'$}}
\DeclareMathOperator*{\argmax}{\arg\max}
\DeclareMathOperator*{\sums}{\sum\sum}
\newtheorem{theorem}{Theorem}
\newtheorem{lemma}[theorem]{Lemma}
\newcommand{\num}[1]{\text{\footnotesize~~#1:}~~}

\journal{European Journal of Operational Research}

\begin{document}

\begin{frontmatter}

\title{On Solving Manufacturing Cell Formation via Bicluster Editing}

\author[uff]{Rian G. S. Pinheiro}\ead{rian.gabriel@ic.uff.com}
\author[uff]{Ivan C. Martins}\ead{imartins@ic.uff.br}
\author[uff]{F\'abio Protti}\ead{fabio@ic.uff.br}
\author[uff]{Luiz S. Ochi}\ead{satoru@ic.uff.br}
\author[uff]{Luidi G. Simonetti}\ead{luidi@ic.uff.br}
\author[ufpb]{Anand Subramanian\corref{cor1}}\ead{anand@ct.ufpb.br}
\cortext[cor1]{Corresponding author. Tel. +55 83 3216-7549; Fax +55 83 3216-7179.}
\address[uff]{Fluminense Federal University\\
Niter\'oi, RJ - Brazil}
\address[ufpb]{Federal University of Para\'iba\\
Jo\~ao Pessoa, PB - Brazil}

\newcommand{\Cpp}{C\nolinebreak\hspace{-.05em}\raisebox{.4ex}{\tiny\bf +}\nolinebreak\hspace{-.10em}\raisebox{.4ex}{\tiny\bf +}~}
\newacronym{bgep}{BGEP}{Bicluster Graph Editing Problem}
\newacronym{mcfp}{MCFP}{Manufacturing Cell Formation Problem}
\newacronym{bgeps}{BGEPS}{Bicluster Graph Editing Problem with Size Restriction}
\newacronym{bgepsp}{BGEPS($\lambda$)}{Bicluster Graph Editing Problem with Size Restriction($\lambda$)}
\newacronym{cgep}{CGEP}{Cluster Graph Editing Problem}

\begin{abstract}
This work investigates the Bicluster Graph Editing Problem (BGEP) and how it can be applied to solve the Manufacturing Cell Formation Problem (MCFP). We develop an exact method for the BGEP that consists of a Branch-and-Cut approach combined with a special separation algorithm based on dynamic programming. We also describe a new preprocessing procedure for the BGEP derived from theoretical results on vertex distances in the input graph. Computational experiments performed on randomly generated instances with various levels of difficulty show that our separation algorithm accelerates the convergence speed, and our preprocessing procedure is effective for low density instances. Other contribution of this work is to reveal the similarities  between the BGEP and the MCFP. We show that the BGEP and the MCFP have the same solution space. This fact leads to the proposal of two new exact approaches for the MCFP based on mathematical formulations for the BGEP. Both approaches use the grouping efficacy measure as the objective function. Up to the authors' knowledge, these are the first exact methods that employ such a measure to optimally solve instances of the MCFP. The first approach consists of iteratively running several calls to a parameterized version of the BGEP, and the second is a linearization of a new fractional-linear model for the MCFP. Computational experiments performed on instances of the MCFP found in the literature show that our exact methods for the MCFP are able to prove several previously unknown optima.
\end{abstract}

\end{frontmatter}

\section{Introduction}
The \gls{bgep} is described as follows: given a bipartite graph $G = (U, V, E)$, where $U$ and $V$ are non-empty stable sets of vertices and $E$ is a set of edges linking vertices in $U$ to vertices in $V$, the goal is to transform $G$ into a disjoint union of complete bipartite graphs (or {\it bicliques}) by performing a minimum number of {\it edge editing operations}. Each edge editing operation consists of either removing an existing edge in $E$ or adding to $E$ a new edge between a vertex in $U$ to a vertex in $V$.

In a bipartite graph $G$, a {\em bicluster} is a subgraph of $G$ isomorphic to a biclique. The existence of biclusters  indicates a high degree of similarity between the data (vertices). In particular, a perfectly clustered bipartite graph is called a \textit{bicluster graph}, i.e., a bipartite graph in which each of its connected components is a biclique. Hence, we can alternatively define the goal of the \gls{bgep}, as stated by \citet{Amit2004},  as follows: ``find a minimum number of edge editing operations in order to transform an input bipartite graph into a bicluster graph''.

\begin{figure}[hbt]
\myps
\centering
\subfloat[Instance.]{\includegraphics[width=0.3\textwidth]{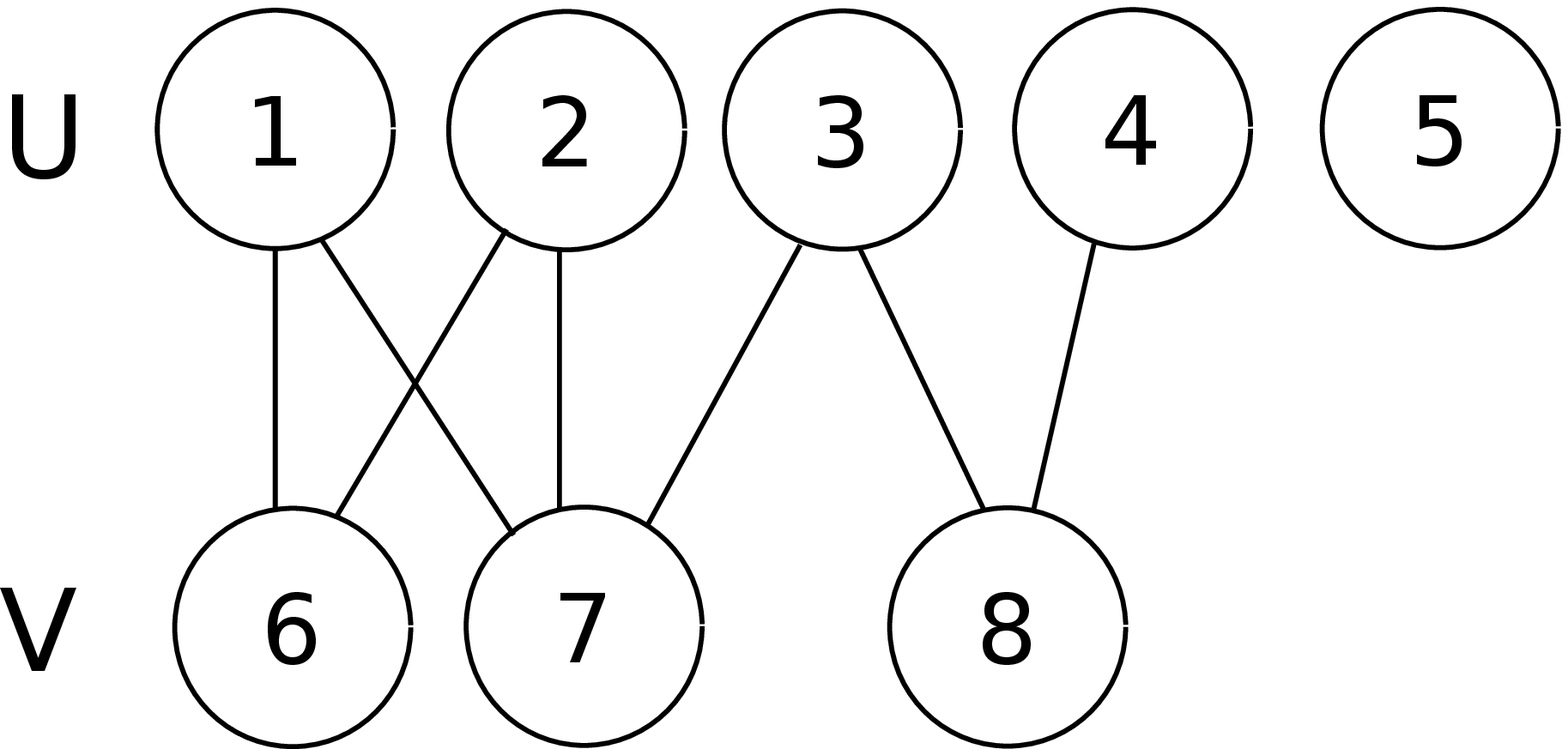}}\hspace{0.1\textwidth}
\subfloat[Solution.]{\includegraphics[width=0.33\textwidth]{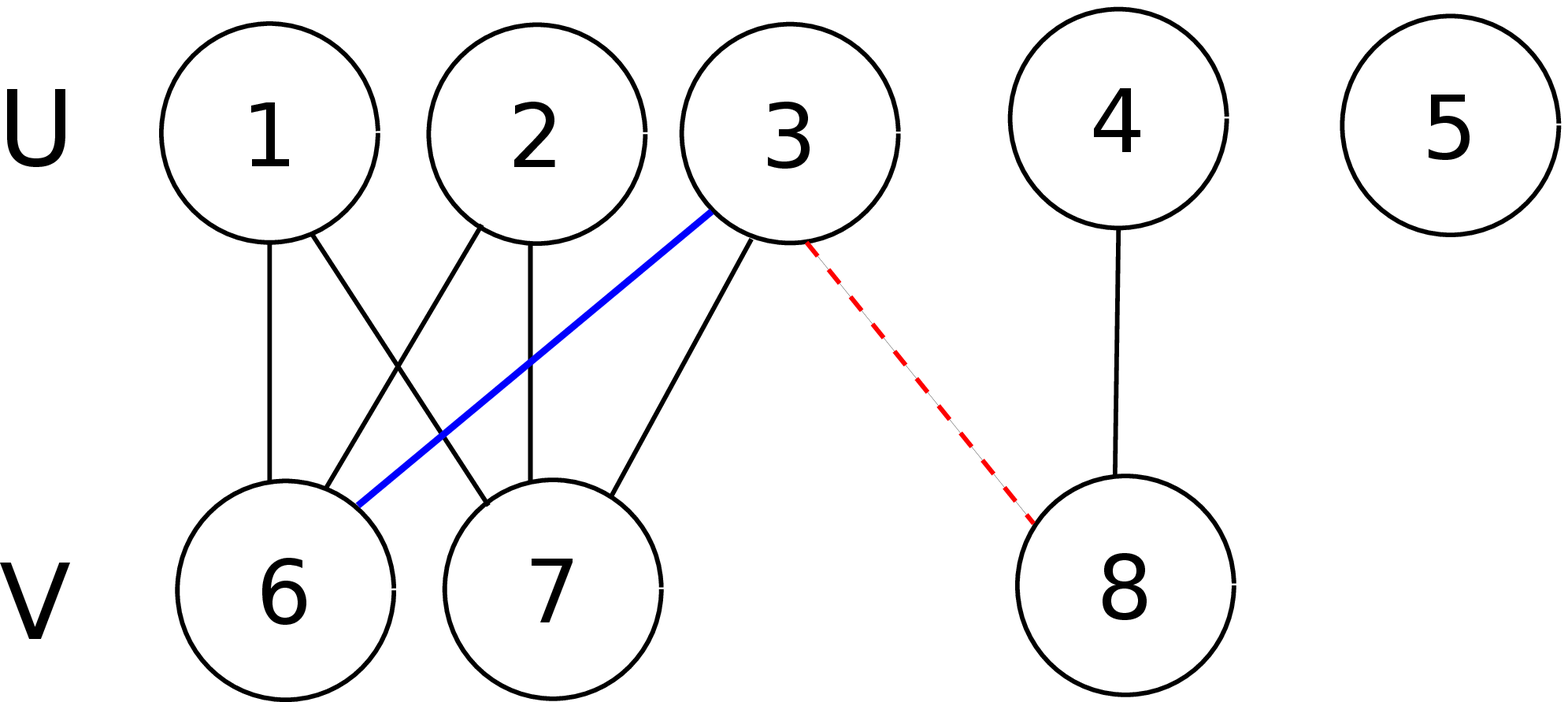}}
\caption{BGEP Example.} \label{fig:ex}
\end{figure}

\autoref{fig:ex} shows an example where adding an edge between vertices $3,6$ and deleting the edge between vertices $3,8$ transforms $G$ into a bicluster graph. Note that this does not correspond to an optimal solution, since $G$ can also be transformed into a bicluster graph by simply removing the edge between $3$ and $7$. We remark that a single vertex is considered as a bicluster (e.g., vertex 5 in \autoref{fig:ex}).

A problem similar to the \gls{bgep} is the \gls{cgep}, first studied by \citet{Gupta1979}. Its goal is to transform $G$ into a disjoint union of complete graphs (cliques). The \gls{cgep} and the \gls{bgep} are important examples of partition problems in graphs.

The concept of biclustering was introduced in the mid-70s by \citet{Hartigan1975}, but its first use appeared in a paper by \citet{Cheng2000}, within the context of Computational Biology. Since then, algorithms for biclustering have been proposed and used in various applications, such as multicast network design~\citep{Faure2007} and analysis of biological data~\citep{Abdullah2006,Bisson2008}.

In Biology, concepts such as co-clustering, two-way clustering, among others, are often used in the literature to refer to the same problem. Matrices are used instead of graphs to represent relationships between genes and characteristics, and their rows and columns represent graph partitions; in this case, the goal is to find significant submatrices having certain patterns. The \gls{bgep} can be used to solve any problem whose goal is to obtain a biclusterization with {\it exclusive} rows and columns, i.e., each gene (characteristic) must be associated with only one submatrix.

\citet{Amit2004} proved the $\mathcal{NP}$-hardness of the \gls{bgep} via a polynomial reduction from the 3-Exact 3-Cover Problem; in the same work, a binary integer programming formulation and an 11-approximation algorithm based on the relaxation of a linear program are described. \citet{Protti2006} proposed an algorithm for the parameterized version of the \gls{bgep} that uses a strategy based on modular decomposition techniques. \citet{Guo2008} developed a randomized 4-approximation algorithm for the \gls{bgep}. More recently, \citet{SousaFilho2012} proposed a GRASP-based heuristic for the \gls{bgep}.

Other important application of the \gls{bgep}, introduced in this work, is related to the \gls{mcfp}. We show that such problems have a high degree of similarity, and that good solutions for the \gls{bgep} are close to good solutions for the \gls{mcfp}. Cellular manufacturing is an application of the Group Technology concept. The goal is to identify and cluster similar parts in order to optimize the manufacturing process. Such a concept was originally proposed by \citet{Flanders1924} and formally described by \citet{Mitrofanov1966} in \citeyear{Mitrofanov1966}. In the early 70s, \citet{Burbidge1971} proposed one of the first techniques for creating a system of cellular manufacturing. Since this work, several approaches have been proposed to the \gls{mcfp}, whose goal is to create the cells in order to optimize the manufacturing process, as described in Section~\ref{sec:cell}.

Our contributions can be summarized as follows. In Section~\ref{sec:branch}, we develop an exact method for the \gls{bgep} consisting of a Branch-and-Cut approach combined with a special separation algorithm based on dynamic programming, and we describe a new preprocessing procedure for the \gls{bgep} derived from theoretical results on vertex distances in the input graph. In Section~\ref{sec:cell}, we explore the similarity between the \gls{bgep} and the MCFP. We show that the \gls{bgep} and the \gls{mcfp} have the same solution space, and due to this fact we propose two new exact approaches for the \gls{mcfp} based on mathematical formulations for the \gls{bgep}. Both approaches use the grouping efficacy measure as the objective function. Up to the authors' knowledge, these are the first exact methods that employ such a measure to optimally solve instances of the \gls{mcfp}. The first approach (Section 3.3) consists of iteratively running several calls to a parameterized version of the \gls{bgep}, and the second (Section 3.4) is a linearization of a new fractional-linear model for the \gls{mcfp}. In Section~\ref{sec:results}, we apply our Branch-and-Cut method for the \gls{bgep} to randomly generated BGEP instances with various levels of difficulty. Experimental results show that our separation algorithm is able to accelerate the convergence speed, and our preprocessing procedure for the \gls{bgep} is effective for low density instances. In addition, computational experiments are performed on instances of the \gls{mcfp} found in the literature. Our exact methods for the \gls{mcfp} are able to prove several previously unknown optima. Section~\ref{sec:conclusions} contains our conclusions.

\section{Branch-and-Cut Approach for the BGEP}\label{sec:branch}

A mathematical model for the \gls{bgep} is described in \citet{Amit2004}. It relies on the simple fact that the graph $P_4$ (a path with four vertices, shown in \autoref{fig:p4}) is a forbidden induced subgraph for a bicluster graph. More precisely, for a bipartite graph $G$, $G$ is a bicluster graph if and only $G$ does not contain $P_4$ as an induced subgraph.

\begin{figure}[hbt]
\centering
 \begin{tikzpicture}
 \node(i) at (0,0)[shape=circle,draw=black]{$i$}; \node(k) at (0,-1)[shape=circle,draw=black]{$k$};
 \node(l) at (4,0)[shape=circle,draw=black]{$l$}; \node(j) at (4,-1)[shape=circle,draw=black]{$j$};
 \draw (i) to (l); \draw (l) to (k); \draw (k) to (j);
 \end{tikzpicture}
 \caption{Graph $P_4$.}
 \label{fig:p4}
\end{figure}
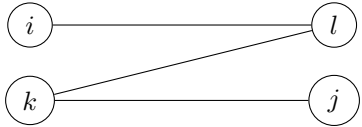

The formulation proposed in \citet{Amit2004} is as follows:

\begin{align}
\label{fo} \min &\quad\sum_{+(ij)}{(1-y_{ij})} + \sum_{-(ij)}{y_{ij}} &\\
\label{p4} \text{s.t.} &\quad  y_{il} + y_{kj} + y_{kl} \leq 2 + y_{ij}	 &\forall \ i\neq k \in U ~\mathrm{and}~ j\neq l \in V \\
\label{int}	&\quad y_{ij} \in \{0,1\} &\forall \ i \in U ~\mathrm{and}~ j \in V
\end{align}

\noindent where: (a) $y_{ij}$ are binary variables such that $y_{ij}=1$ if and only if the solution contains edge $ij$; (b) $+(ij) = \{ij \mid ij \in E\}$ is the set of edges; (c) $-(ij) = \{ij \mid ij \notin E\}$ the set of non-edges. The objective function \eqref{fo} counts how many edge editing operations are made. The first and second sums represent the number of edge deletions and edge additions, respectively. Constraints \eqref{p4} eliminate induced subgraphs isomorphic to $P_4$. Constraints \eqref{int} define the domain of the variables.

\subsection{Separation Algorithm}\label{sec:sep}
Note that the number of constraints \eqref{p4} in the above formulation is $|U|^2 |V|^2$ and therefore it is not advisable to consider all these constraints \textit{a priori}. This becomes computationally expensive for exact methods, especially when dealing with large instances. Alternatively, we start without such constraints and we add them in a cutting plane fashion as they are violated according to the separation algorithm described below:

\begin{algorithm}
\caption{Separation algorithm}
\label{alg:separation}
\vspace*{.5\baselineskip}
\begin{tabbing}
\num{1}{\bf procedure} Separation(relaxation $y^*$)\\
\\
\num{2}\qquad \=$\forall i \in U \text{~and~} j \in V$ \hspace{2cm}  \=  $d_{ij}^{1}~$ \= $=$ \= $ y^*_{ij}$ \hspace{4cm} \= \\
\\
\num{3}\>$\forall i,k \in U$ \> ${d_{ik}^{2}}\!'$ \> = \> $\max\limits_{l \in V }\{d_{il}^{1} + d_{kl}^{1}\}$ \> $l'_{ik} = \argmax\limits_{l \in V }\{d_{il}^{1} + d_{kl}^{1}\}$\\
\\
\num{4}\>$\forall i,k \in U$ \> ${d_{ik}^{2}}\!''$ \> = \> $\max\limits_{l \in V \setminus \{l'_{ik}\}}\{d_{il}^{1} + d_{kl}^{1}\}$ \> $l''_{ik} = \argmax\limits_{l \in V \setminus \{l'_{ik}\}}\{d_{il}^{1} + d_{kl}^{1}\}$\\
\\
\num{5}\>$\forall i,k\in U \ (i\neq k) \ \text{~and~} j \in V$ \> $d_{ikj}^{2}$ \> = \>$ \begin{cases}
{d_{ik}^{2}}\!', & \text{if~} j \neq l'_{ik}\\
{d_{ik}^{2}}\!'', & \text{if~} j = l'_{ik}
\end{cases}$ \>
$l_{ik} =\begin{cases}
l'_{ik}, & \text{if~} j \neq l'_{ik}\\
l''_{ik}, & \text{if~} j = l'_{ik}
\end{cases}$\\
\\
\num{6}\>$\forall i \in U \text{~and~} j \in V$ \> $d_{ij}^{3}$ \> = \> $ \max\limits_{k \in U \setminus \{i\}}\{d_{ikj}^{2} + d_{kj}^{1}\} $ \> $k_{ij} = \argmax\limits_{k \in U \setminus \{i\}}\{d_{ikj}^{2} + d_{kj}^{1}\}$\\
\\
\num{7}\>$\forall i \in U \text{~and~} j \in V$ \> $\text{if}~ d_{ij}^{3}- d_{ij}^1 > 2$ \text{~then add cut~} $y_{il} + y_{kj} + y_{kl} \leq 2 + y_{ij}$\\
\> \> $(\text{~for~} k=k_{ij} \text{~and~} l= l_{i k_{ij}})$ \\
\\
\num{8}{\bf end procedure}
\end{tabbing}
\end{algorithm}

Algorithm~\ref{alg:separation} works with a linear relaxation as input. Its main objective is to find the {\em most violated constraint \eqref{p4}} for each pair $(i,j)$ of vertices, and then add it to the model. The idea is to use an auxiliary complete bipartite graph $G'(U,V,E)$ where each edge $ij$ has a nonnegative weight $w_{ij}$; the weights are defined according the values $y^*_{ij}$ obtained by the linear relaxation, i.e., $w_{ij}=y^*_{ij}$. Note that an edge may have a zero weight.

After the construction of $G'$, a dynamic programming approach is used to find the constraints. It calculates the values $d_{ij}^{s}$, where $d_{ij}^{s}$ is the {\em maximum cost} between $i$ and $j$ considering paths with $s-1$ internal vertices. This is explained below in detail.

In line 2, $d_{ij}^{1}$ is initialized with the value of the linear relaxation $y_{ij}^*$. In line 3, for each pair $(i,k)$ of vertices, the maximum cost ${d_{ik}^{2}}\!'$ between them considering paths with a single internal vertex is calculated; also, the internal vertex by which such a cost is achieved is saved in variable $l'_{ik}$. Line 4 is similar to line 3, but instead of calculating the maximum cost, it calculates the second maximum cost. Line 5 verifies, for all $i,k \in U$ $(i\neq k)$ and $j \in V$, if $j$ belongs to the maximum cost path, and chooses to use the maximum cost path or the second maximum cost path. In this case, $d_{ikj}^{2}$ represents the maximum cost between $i$ and $k$ using a path that avoids $j$, and $l_{ik}$ stores the corresponding internal vertex. Line 6 calculates the maximum cost between $i$ and $j$ using two internal vertices, and stores it in $d_{ij}^{3}$; it represents the ``$P_4$ of maximum cost''; variable $k_{ij}$ saves the internal vertex $k$. Finally, in line 7,  for each pair $(i,j)$, if the constant is violated ($d_{ij}^{3} - d_{ij}^{1} > 2$) then the cut $y_{il} + y_{kj} + y_{kl} \leq 2 + y_{ij}$ for $k=k_{ij}$ and $l= l_{i k_{ij}}$ is added to the model.

\subsection{Preprocessing Procedure}\label{sec:pre}

In this section, we propose a preprocessing procedure to fix variables and/or generate new constraints to the \gls{bgep}. The procedure is a direct application of \autoref{teo:distance}, presented below. New generated constraints will be added to the Branch-and-Cut algorithm.

\begin{theorem}\label{teo:distance}
Let $a,b$ be vertices of a bipartite graph $ G(V,U,E)$, and let $d(a,b)$ be the distance between $a$ and $b$ in $G$. If $d(a,b)\geq 4$ then there is an optimal solution in which $a,b$ belong to distinct biclusters.
\end{theorem}

\begin{proof}
The proof consists of showing that, when $d(a,b) \geq 4$, the cost of keeping $a$ and $b$ in the same bicluster is greater than or equal to the cost of keeping them in distinct biclusters. The following notation is useful for the proof. Let $X \subseteq V$ and $Y \subseteq U$. Let $rem(X,Y)= |\{ xy \in E \mid x\in X ~\text{and}~ y\in Y\}|$; informally, $rem(X,Y)$ is the cost of removing all edges in $E$ between $X$ and $Y$. Also, let $add(X,Y)= |X||Y|-rem(X,Y)$, i.e., $add(X,Y)$ is the cost of adding all the missing edges between $X$ and $Y$ in order to create a bicluster $B$ with vertex set $V(B)=X\cup Y$. If $X=\{x\}$, we simply write $add(x,Y)$ and $rem(x,Y)$ instead of $add(\{x\},Y)$ and $rem(\{x\},Y)$, and similarly if $Y=\{y\}$. Denote by $N(a)$ the neighborhood of $a$, and let $N^2(a) = \{v \in V \cup U \mid d(a,v) = 2\}$.

For the case $d(a,b) = \infty$, note that $a$ and $b$ lie in distinct connected components of $G$, and therefore will belong to distinct biclusters in any optimal solution. Now assume that $d(a,b)<\infty$ and there is an optimal solution $G^*$ in which vertices $a,b$ belong to a bicluster $B$ with vertex set $V(B)=X\cup Y$, where $X \subseteq V$ and $Y \subseteq U$. We analyze two cases: $a,b\in X$ (Case 1) and $a\in X, b\in Y$ (Case 2). Case 1 is divided in two sub-cases: $d(a,b)>4$ (Case 1a) and $d(a,b) = 4$ (Case 1b).

Let $N_a = N(a)\cap V(B)$, $N_b = N(b)\cap V(B)$, $N^2_a = N^2(a)\cap V(B)$, and $N^2_b = N^2(b)\cap V(B)$.

\begin{figure}[htb]
\myps
\centering
\subfloat[Case 1a]
{\includegraphics[width=0.3\textwidth]{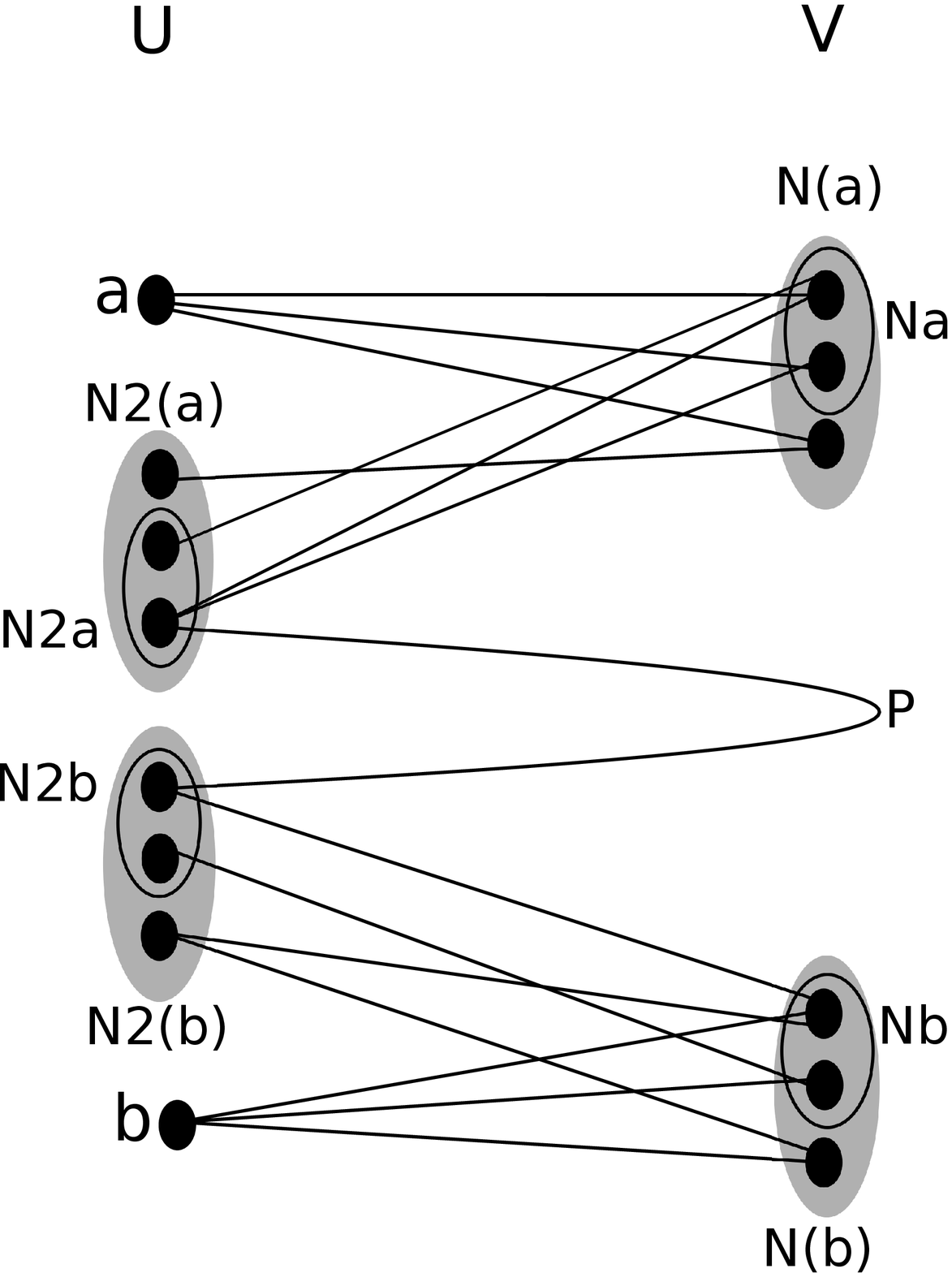}\label{fig:case1a}}\hspace{0.1\textwidth}
\subfloat[Case 1b]
{\includegraphics[width=0.3\textwidth]{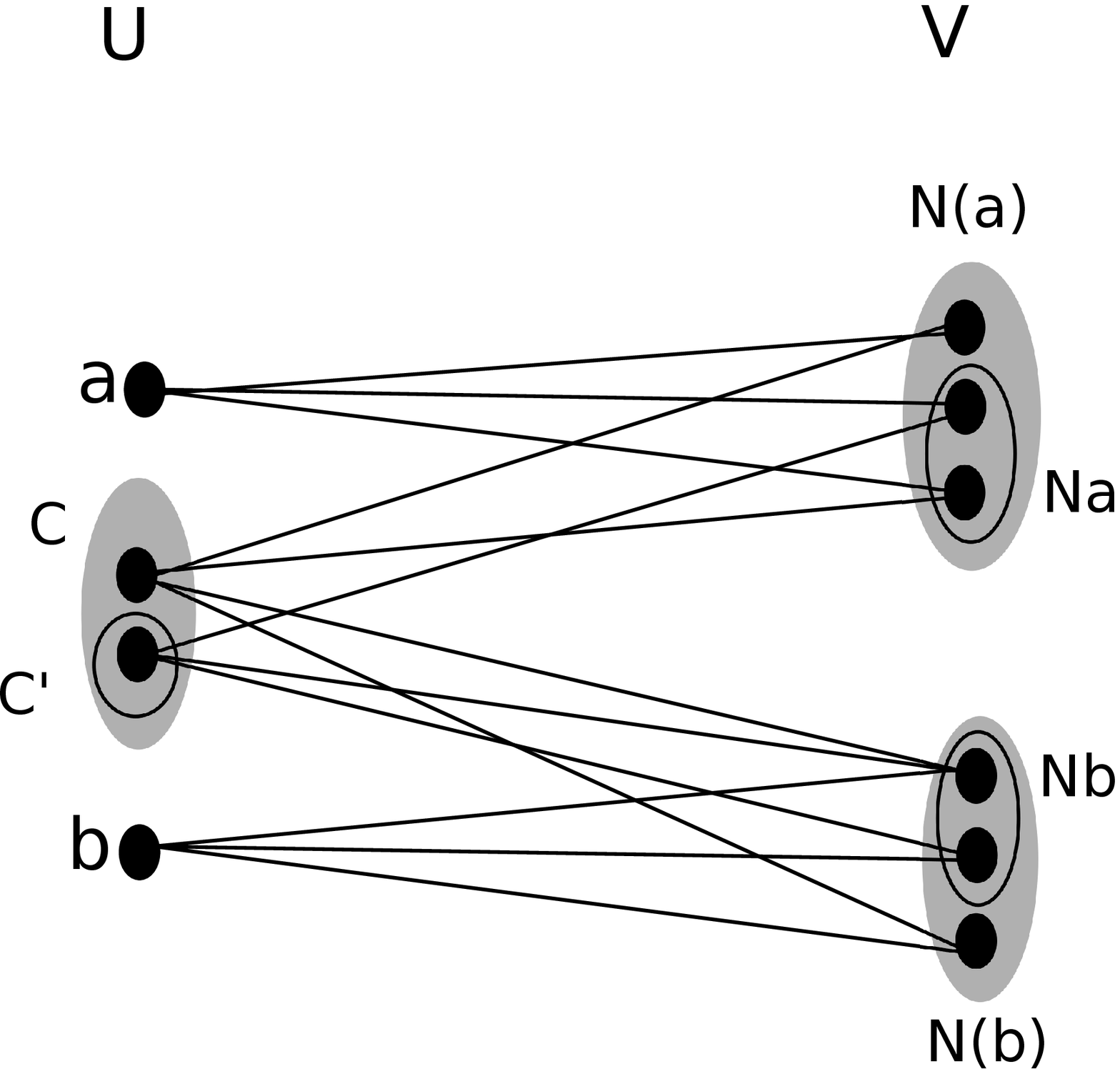}\label{fig:case1b}} \\
\subfloat[Case 2]
{\includegraphics[width=0.3\textwidth]{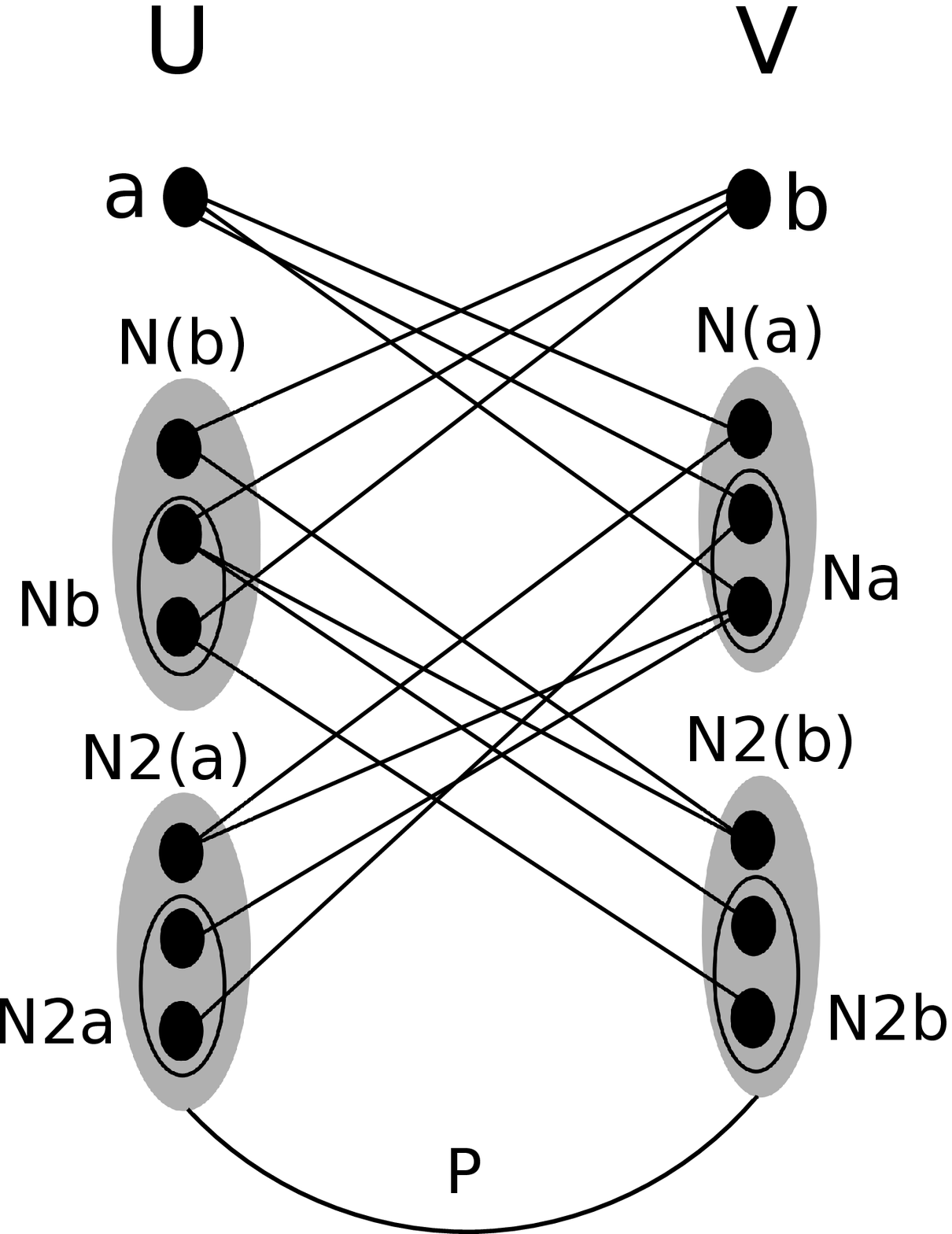}\label{fig:case2}}
\caption{Cases of Theorem~\ref{teo:distance}} \label{fig:tes}
\end{figure}

\medskip

\noindent {\bf Case 1a:} \ In this case, note that $N^2(a) \cap N^2(b) = \emptyset$. Figure~\ref{fig:case1a} represents this situation, where $P$ is a path between $N^2(a)$ and $N^2(b)$. Let $X'\subseteq X$ and $Y'\subseteq Y$ defined as follows:

$$X'= \{a,b\} \cup N^2_a \cup N^2_b \cup (V(P) \cap U) ~\text{and}~ Y'= N_a \cup N_b \cup (V(P) \cap V),$$

\noindent where $V(P)$ is the subset of vertices in path $P$. Since $X'$ and $Y'$ are completely connected by edges in $G^*$, the following cost $c_1$ is associated with the structure represented in Figure~\ref{fig:case1a}:

\begin{small}
\[\begin{array}{lll}
c_1 &=& rem(a,N(a)\setminus N_a) + rem(N_a, N^2(a) \setminus N^2_a) \ + \ rem(b,N(b)\setminus N_b) + rem(N_b, N^2(b) \setminus N^2_b) \ +\\
\\
& &add(N^2_a,N_a) + add(N^2_b,N_b) + rem(N^2_a, V \setminus Y') \ + \ rem(N^2_b, V \setminus Y') \ +\\
\\
& &add(a,N_b) + add(b,N_a) + add(N^2_a,N_b) + add(N^2_b, N_a) \ +\\
\\
& &add(X', V(P) \cap V) + add(V(P) \cap U,N_a \cup N_b ).
\end{array}\]
\end{small}

Now consider the following subsets: $X'_1= \{a\}\cup N^2_a, \ Y'_1= N_a, \ X'_2= \{b\}\cup N^2_b, \ Y'_2= N_b$. Let $G^{**}$ be another solution with distinct biclusters $B_1$ and $B_2$ such that $X'_1 \cup Y'_1 \subseteq V(B_1)$ and $X'_2 \cup Y'_2 \subseteq V(B_2)$. The cost $c_2$ associated with this new situation is given by:

\begin{small}
\[\begin{array}{lll}
c_2 &=& rem(a,N(a)\setminus N_a) + rem(N_a, N^2(a) \setminus N^2_a) \ + \ rem(b,N(b)\setminus N_b) + rem(N_b, N^2(b) \setminus N^2_b) \ +\\
\\
& &add(N^2_a,N_a) + add(N^2_b,N_b) + rem(N^2_a, V \setminus Y') \ + \ rem(N^2_b, V \setminus Y').
\end{array}\]
\end{small}

Since $c_1\geq c_2$, the cost of keeping $a$ and $b$ in distinct biclusters is not greater than the cost of keeping them in the same bicluster.

\medskip

\noindent {\bf Case 1b:} \ If $d(a,b)=4$ then $a$ and $b$ belong to the same part, say $U$, and $N^2(a) \cap N^2(b) = C\neq\emptyset$. Let $C'=C\cap V(B)$. Figure~\ref{fig:case1b} depicts the subsets involved in Case 1b.

Let $X'\subseteq X$ and $Y'\subseteq Y$ defined as $X'= \{a,b\} \cup C'$ and $Y'= N_a \cup N_b.$ Again, since $X'$ and $Y'$ are completely connected by edges in $G^*$, the cost $c_1$ associated with the structure in Figure~\ref{fig:case1b} is given by:

\begin{small}
\[\begin{array}{lll}
c_1 &=& add(b,N_a) + add(a,N_b) + rem(N_a,N^2(a)\setminus C') + rem(N_b,N^2(b) \setminus C') + add(C',N_a) + add(C',N_b) \ +\\
\\
& &rem(a,N(a)\setminus N_a) + rem(b,N(b)\setminus N_b) + rem(C',V \setminus (N_a \cup N_b)).
\end{array}\]
\end{small}

Assume without loss of generality that $|N_b|\leq |N_a|$. Consider now the subsets $X'_1= \{a\}\cup C', \ Y'_1= N_a\cup N_b, \ X'_2= \{b\}, \ Y'_2=\emptyset$. If $G^{**}$ is another solution with distinct biclusters $B_1$ and $B_2$ such that $X'_1 \cup Y'_1 \subseteq V(B_1)$ and $X'_2 \cup Y'_2 \subseteq V(B_2)$, the cost $c_2$ associated with $G^{**}$ is:

\begin{small}
\[\begin{array}{lll}
c_2 &=& add(a,N_b) + rem(N_a,N^2(a)\setminus C') + rem(N_b,N^2(b) \setminus C') + add(C',N_a) + add(C',N_b) \ +\\
\\
& &rem(a,N(a)\setminus N_a) + rem(C',V \setminus (N_a \cup N_b)) + rem(b,N(b)).
\end{array}\]
\end{small}

The difference of the costs is given by:

\begin{small}
\[\begin{array}{lll}
c_1-c_2 &=& add(b,N_a) + rem(b,N(b)\setminus N_b) - rem(b,N(b))\\
\\
&=&|N_a| + |N(b)| - |N_b| - |N(b)| = |N_a| - |N_b|.
\end{array}\]
\end{small}

Therefore $c_1-c_2\geq 0$, i.e., keeping $a$ and $b$ in distinct biclusters is not more costly.

\medskip

\noindent {\bf Case 2:} \ In this case, $a$ and $b$ belong to different parts. Assume $a\in U$ and $b\in V$, as shown in Figure~\ref{fig:case2}. Let $X'\subseteq X$ and $Y'\subseteq Y$ defined as $X'=\{a\} \cup N_b \cup N_a^2$ and $Y'= \{b\} \cup N_a \cup N^2_b$. The cost $c_1$ corresponding to the biclique with vertex set $X'\cup Y'$ is:

\begin{small}
\[\begin{array}{lll}
c_1 &=& add(a,b) + add(a,N^2_b) + add(b,N^2_a) + add(N_a,N_b) + add(N_a,N^2_a) + add(N_b,N^2_b) + add(N^2_a,N^2_b) \ +\\
\\
& &add(X', V(P) \cap V) + add(V(P) \cap U, Y') + rem(a,N(a)\setminus N_a) + rem(b,N(b)\setminus N_b) \ +\\
\\
& &rem(N^2_a,V \setminus  N_a ) + rem(N^2_b,U \setminus N_b) + rem(N_a,N^2(a)\setminus N^2_a) + rem(N_b,N^2(b)\setminus N^2_b).
\end{array}\]
\end{small}

Now let $X'_1= \{a\}\cup N^2_a, \ Y'_1= N_a, \ X'_2= N_b, \ Y'_2=\{b\}$. Also, let $G^{**}$ be another solution with distinct biclusters $B_1$ and $B_2$ such that $X'_1 \cup Y'_1 \subseteq V(B_1)$ and $X'_2 \cup Y'_2 \subseteq V(B_2)$. The cost $c_2$ associated with $G^{**}$ is:

\begin{small}
\[\begin{array}{lll}
c_2 &=& rem(a,N(a)\setminus N_a) + rem(b,N(b)\setminus N_b) + rem(N_a,N^2(a)\setminus N^2_a) + rem(N_b,N^2(b)\setminus N^2_b) \ +\\
\\
& &rem(N^2_a,V \setminus  N_a) + rem(N^2_b,U \setminus N_b) + add(N_a,N^2_a) + add(N_b,N^2_b).
\end{array}\]
\end{small}

Again, the separation into two biclusters is not more costly.

\medskip

For each case above, we have shown that there is another solution in which $a$ and $b$ belong to distinct biclusters and whose cost is not greater. Then the theorem follows.
\end{proof}

Based on the above theorem, after computing the distance between each pair of vertices, variables can be fixed and cuts can be generated as follows: if vertices $i$ and $j$ are in different parts and $d(i,j)>3$ then variable $y_{ij}$ is set to 0. Otherwise, $i$ and $j$ are in the same part, say $U$, and the cut $y_{ik} + y_{jk} \leq 1 $ will be generated for every $k \in V$.

\subsection{Instances for the BGEP}\label{sec:instances_bgep}

To the best of our knowledge, no public sites contain instances for the \gls{bgep}. Thus, to evaluate the algorithms proposed in this work, we create random bipartite instances using the $\mathbb G(m,n,p)$ model, also known as binomial model, which is a particular case of the model proposed by \citet{Gilbert1959}. A bipartite instance $G(U,V,E)$ is created such that $|U| = m$, $|V| = n$, and each potential edge of $E$ is created with probability $p$, independently of the other edges.

\section{Application to Manufacturing Cell Formation}\label{sec:cell}

The input of the \gls{mcfp} is given as a binary product-machine matrix where each entry $(i,j)$ has value 1 if product $j$ is manufactured by machine $i$, and 0 otherwise. Any feasible solution of the \gls{mcfp} consists of a collection of product-machine cells, where every product (or machine) is allocated to exactly one cell. Hence, in each cell $C$, machines allocated to it are exclusively dedicated to manufacture products in $C$. In an ideal solution of the \gls{mcfp}, in each cell there must be a high similarity between the products and machines allocated to it. \autoref{fig:cell_example} shows an example of the \gls{mcfp} solved as a block diagonalization problem. Note that a solution for the \gls{mcfp} is obtained by a permutation of rows/columns of the input matrix together with a cell assignment for products and machines. In \autoref{fig:cell_example}, products $P_1, P_3, P_7$ and machines $M_2, M_3, M_5$ are gathered to form a cell, while the remaining products/machines form another cell.

\begin{figure}[hbt]
\begin{minipage}[b]{0.48\linewidth}\centering
 \begin{tabular}{c!{\VRule[2pt]}ccccc!{\VRule[2pt]}}
\multicolumn{1}{c}{} & $M_1$ & $M_2$ & $M_3$ & $M_4$ & \multicolumn{1}{c}{$M_5$}\\ \cmidrule[2pt]{2-6}
$P_1$ &  0 & 1 & 1 & 0 & 1 \\
$P_2$ &  1 & 0 & 0 & 1 & 0 \\
$P_3$ &  0 & 1 & 1 & 0 & 0 \\
$P_4$ &  1 & 0 & 0 & 1 & 0 \\
$P_5$ &  1 & 0 & 0 & 0 & 1 \\
$P_6$ &  1 & 0 & 1 & 1 & 0 \\
$P_7$ &  0 & 0 & 1 & 0 & 1 \\
  \cmidrule[2pt]{2-6}
 \end{tabular}
 \end{minipage}
$\Rightarrow$
\begin{minipage}[b]{0.48\linewidth}\centering
\begin{tabular}{cccc!{\VRule[2pt]}cc}
\multicolumn{1}{c}{} & $M_2$ & $M_3$ & \multicolumn{1}{c}{$M_5$} & $M_1$ & \multicolumn{1}{c}{$M_4$}\\  \cmidrule[2pt]{2-4}
\multicolumn{1}{c!{\VRule[2pt]}}{$P_1$} &  1 & 1 & 1 & 0 & 0 \\
\multicolumn{1}{c!{\VRule[2pt]}}{$P_3$} &  1 & 1 & 0 & 0 & 0 \\
\multicolumn{1}{c!{\VRule[2pt]}}{$P_7$} &  0 & 1 & 1 & 0 & 0 \\ \cmidrule[2pt]{2-6}
$P_2$ &  0 & 0 & 0 & 1 & \multicolumn{1}{c!{\VRule[2pt]}}{1} \\
$P_4$ &  0 & 0 & 0 & 1 & \multicolumn{1}{c!{\VRule[2pt]}}{1} \\
$P_5$ &  0 & 0 & 1 & 1 & \multicolumn{1}{c!{\VRule[2pt]}}{0} \\
$P_6$ &  0 & 1 & 0 & 1 & \multicolumn{1}{c!{\VRule[2pt]}}{1} \\ \cmidrule[2pt]{5-6}
 \end{tabular}
 \end{minipage}
 \caption{MCFP example.}
 \label{fig:cell_example}
\end{figure}

Among several measures of performance used as objective functions for the \gls{mcfp}, the most used in literature is the {\em grouping efficacy} $\mu$, defined as:

\label{sec:of_cell}
\begin{equation} \label{eq:of_cell}
 \mu = \frac{N_1 - N^{out}_{1}}{N_1 + N^{in}_{0}},
\end{equation}
 where $N_1$ is the total number of 1's in the input matrix, and $N^{out}_{1}$ ($N^{in}_{0}$) is the total number of 1's outside (respectively, inside) diagonal blocks in the solution matrix.

\subsection{The size of the cells}

Some works define a minimum value for the size of the cells. For instance, in  \citep{Chandrasekharan1987,Srinivasan1991,Goncalves2004}, cells with less than two products or machines are not allowed; such cells are called \textit{singletons}. However, there is no consensus with respect to the size of the cells. Others studies do not consider any size constraint, allowing the existence of empty cells, such as the work by \citet{Pailla2010}. An example of a solution with an empty cell is shown in \autoref{fig:ex_empty}.

This work deals with two versions of the \gls{mcfp} found in the literature:
\begin{enumerate}
\item unrestricted version, allowing singletons and empty cells;
\item with restrictions (minimum size $2 \times 2$ for each cell).
\end{enumerate}

\begin{figure}[hbt]
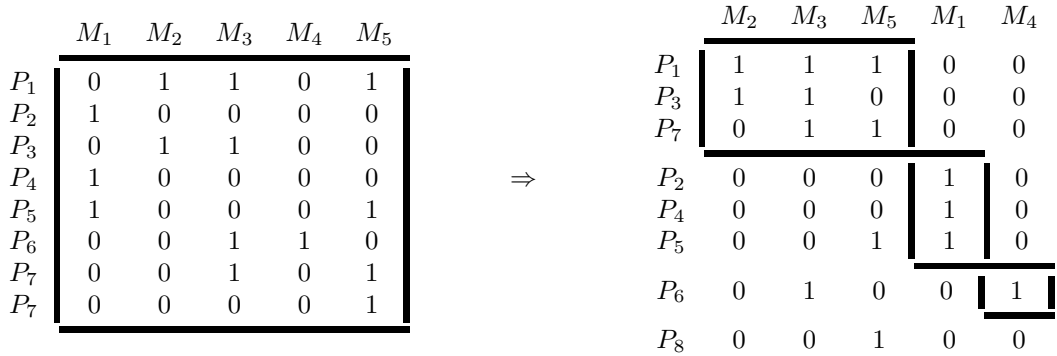

\begin{minipage}[b]{0.48\linewidth}\centering
 \begin{tabular}{c!{\VRule[2pt]}ccccc!{\VRule[2pt]}}
\multicolumn{1}{c}{} & $M_1$ & $M_2$ & $M_3$ & $M_4$ & \multicolumn{1}{c}{$M_5$}\\ \cmidrule[2pt]{2-6}
$P_1$ &  0 & 1 & 1 & 0 & 1 \\
$P_2$ &  1 & 0 & 0 & 0 & 0 \\
$P_3$ &  0 & 1 & 1 & 0 & 0 \\
$P_4$ &  1 & 0 & 0 & 0 & 0 \\
$P_5$ &  1 & 0 & 0 & 0 & 1 \\
$P_6$ &  0 & 0 & 1 & 1 & 0 \\
$P_7$ &  0 & 0 & 1 & 0 & 1 \\
$P_7$ &  0 & 0 & 0 & 0 & 1 \\
  \cmidrule[2pt]{2-6}
 \end{tabular}
 \end{minipage}
$\Rightarrow$
\begin{minipage}[b]{0.48\linewidth}\centering
\begin{tabular}{cccc!{\VRule[2pt]}cc}
\multicolumn{1}{c}{} & $M_2$ & $M_3$ & \multicolumn{1}{c}{$M_5$} & $M_1$ & \multicolumn{1}{c}{$M_4$}\\  \cmidrule[2pt]{2-4}
\multicolumn{1}{c!{\VRule[2pt]}}{$P_1$} &  1 & 1 & 1 & 0 & 0 \\
\multicolumn{1}{c!{\VRule[2pt]}}{$P_3$} &  1 & 1 & 0 & 0 & 0 \\
\multicolumn{1}{c!{\VRule[2pt]}}{$P_7$} &  0 & 1 & 1 & 0 & 0 \\ \cmidrule[2pt]{2-5}
$P_2$ &  0 & 0 & 0 & 1 & \multicolumn{1}{!{\VRule[2pt]}c}{0} \\
$P_4$ &  0 & 0 & 0 & 1 & \multicolumn{1}{!{\VRule[2pt]}c}{0} \\
$P_5$ &  0 & 0 & 1 & 1 & \multicolumn{1}{!{\VRule[2pt]}c}{0} \\ \cmidrule[2pt]{5-6}
$P_6$ &  0 & 1 & \multicolumn{1}{c}{0} & \multicolumn{1}{c!{\VRule[2pt]}}{0} & \multicolumn{1}{c!{\VRule[2pt]}}{1} \\ \cmidrule[2pt]{6-6}
$P_8$ &  0 & 0 & \multicolumn{1}{c}{1} & 0 & 0 \\
 \end{tabular}
 \end{minipage}
\caption{Example with a singleton and an empty cell.}
\label{fig:ex_empty}
 \end{figure}

\subsection{Similarity between the BGEP and the MCFP}

Given an input for the \gls{mcfp}, we can define an input $G$ for the \gls{bgep} by setting $U$ as the set of products, $V$ as the set of machines, and $E$ as the set of edges such that $ij$ is an edge of $G$ if and only if the entry $(i,j)$ of the input matrix for the \gls{mcfp} has value 1. In addition, a solution for the \gls{bgep} with input $G$ can be transformed into a cell assignment for machines/products.

The two problems (the \gls{bgep} and the \gls{mcfp}) are very similar, but a point to note is that biclusters have no size limitation. Thus, we define a new \gls{bgep} variant, the \gls{bgeps}, to make a more precise correspondence with the \gls{mcfp}. Informally, the \gls{bgeps} is defined by adding size restrictions to the \gls{bgep}: every bicluster in $G$ must now have at least $s_c$ vertices in $U$ and $s_r$ vertices in $V$. In the translation from the \gls{bgeps} to the \gls{mcfp}, $s_r$ is the minimum cell size for rows and $s_c$ the minimum cell size for columns.

We propose the following formulation for the \gls{bgeps}:
\begin{align}
  \min 		&\quad \sum_{+(ij)}{(1-y_{ij})} + \sum_{-(ij)}{y_{ij}} 	&\\
  st	 	&\quad  y_{il} + y_{kj} + y_{kl} \leq 2 + y_{ij}		&\forall i,k \in U ~and~ j,l \in V \\
\label{tamc}	&\quad \sum_{j\in V}{y_{ij}} \geq s_r			& \forall i \in  U \\
\label{taml}	&\quad \sum_{i\in U}{y_{ij}} \geq s_c			& \forall j \in  V \\
		&\quad y_{ij} \in \{0,1\} 					&\forall i \in U ~and~ j \in V.
\end{align}
Lemma~\ref{lem:solution} makes the correspondence between the \gls{bgeps} and the \gls{mcfp}.

\begin{lemma}\label{lem:solution}
There is a one-to-one correspondence from the feasible solution set of the \gls{bgeps} to the feasible solution set of the \gls{mcfp}. In addition, for every pair of corresponding feasible solutions, the sizes of biclusters and cells are preserved.
\end{lemma}

\begin{proof}
Consider an instance of the \gls{bgeps} consisting of a bipartite graph $G$ with parts $U$ and $V$. It is easy to see that it corresponds to an instance $M$ of the \gls{mcfp}: Figures~\ref{fig:tras:ins_bgep} and~\ref{fig:tras:ins_mcfp} show an example of a bipartite graph $G$ and a corresponding product-machine matrix $M$. Let $\mathcal B$ and $\mathcal C$ be the feasible solution sets of the \gls{bgeps} and the \gls{mcfp}, respectively, associated with $G$ and $M$.

Let $f : \mathcal B \rightarrow \mathcal C$ be the transformation of a solution $G'$ in $\mathcal B$ to a solution $M'$ in $\mathcal C$ defined as follows. If vertices $i\in U$ and $j\in V$ belong to the same bicluster in $G'$ then product $i$ and machine $j$ are gathered inside the same cell in $M'$, as shown in Figures~\ref{fig:tras:sol_bgep} and~\ref{fig:tras:sol_mcfp}. Thus, each solution in $\mathcal B$ is uniquely mapped into one solution in $\mathcal C$.

Similarly, let $g : \mathcal C \rightarrow \mathcal B$ be the inverse transformation of $f$ that maps each solution $M'$ in $\mathcal C$ into a solution $G'$ in $\mathcal B$, as follows: if a product $i$ and a machine $j$ are in the same cell then the corresponding vertices $i$ and $j$ belong to the same bicluster in $G'$.

Since $f$ and $g$ are injective functions, there is a one-to-one correspondence between $\mathcal B$ and $\mathcal C$. Moreover, if $G'$ and $M'$ are corresponding feasible solutions, it is easy to see that a bicluster $B$ in $G'$ corresponds to a cell containing exactly $|V(B)\cap U|$ products, $|V(B)\cap V|$ machines, and $|E(B)|$ entries; also, a cell with $pm$ entries in $M'$ corresponds to a bicluster with exactly $p+m$ vertices and $pm$ edges.
\end{proof}

An optimal solution $G^*$ of the \gls{bgep} does not necessarily correspond to an optimal solution of the \gls{mcfp}, because the objective functions are different; however, $G^*$ corresponds to a feasible solution of the \gls{mcfp}. For example, Figure~\ref{fig:tras:sol_bgep} shows an optimal solution of the \gls{bgep}, but in Figure~\ref{fig:tras:sol_mcfp} the corresponding solution of the \gls{mcfp} is not optimal. This happens because an edge deletion, informally, corresponds to a `1' outside cells, and an edge addition corresponds to a `0' inside a cell. That is, for the \gls{bgep}, additions and deletions have the same weight, but for the \gls{mcfp}, a `0 inside' is preferable  than a `1 outside' (using the objective funtion~\eqref{eq:of_cell}). In Figure \ref{fig:tras:ins_bgep}, for instance, adding an edge between vertices $5$ and $8$ is better than deleting the edge between $4$ and $9$, in terms of the corresponding solutions of the \gls{mcfp}.

\begin{figure}[hbtp]
\myps
\centering
\subfloat[\gls{bgeps} instance.]{\includegraphics[width=0.33\textwidth]{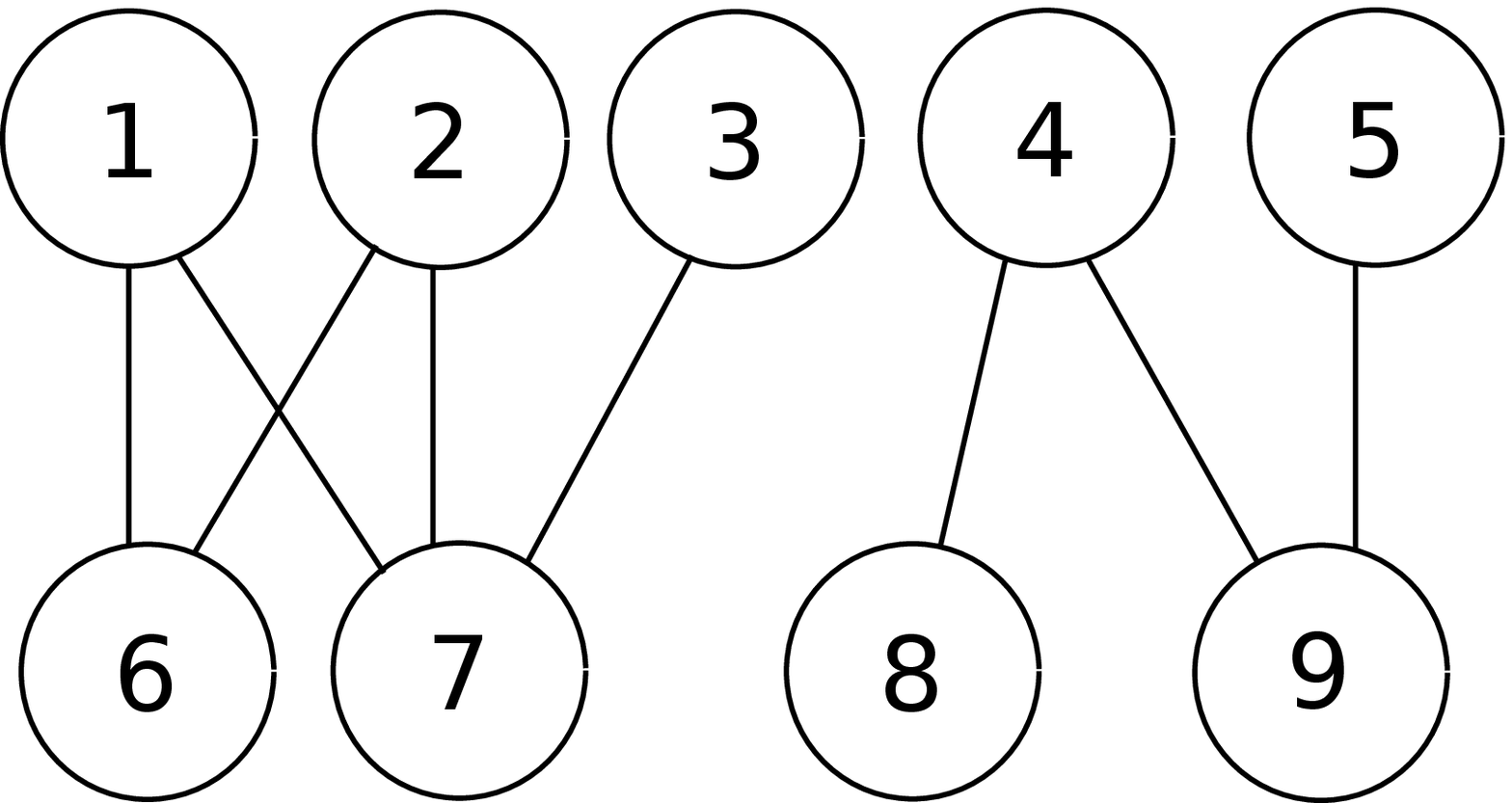}\label{fig:tras:ins_bgep}}\hspace{0.1\textwidth}
\subfloat[\gls{bgeps} solution.]{\includegraphics[width=0.33\textwidth]{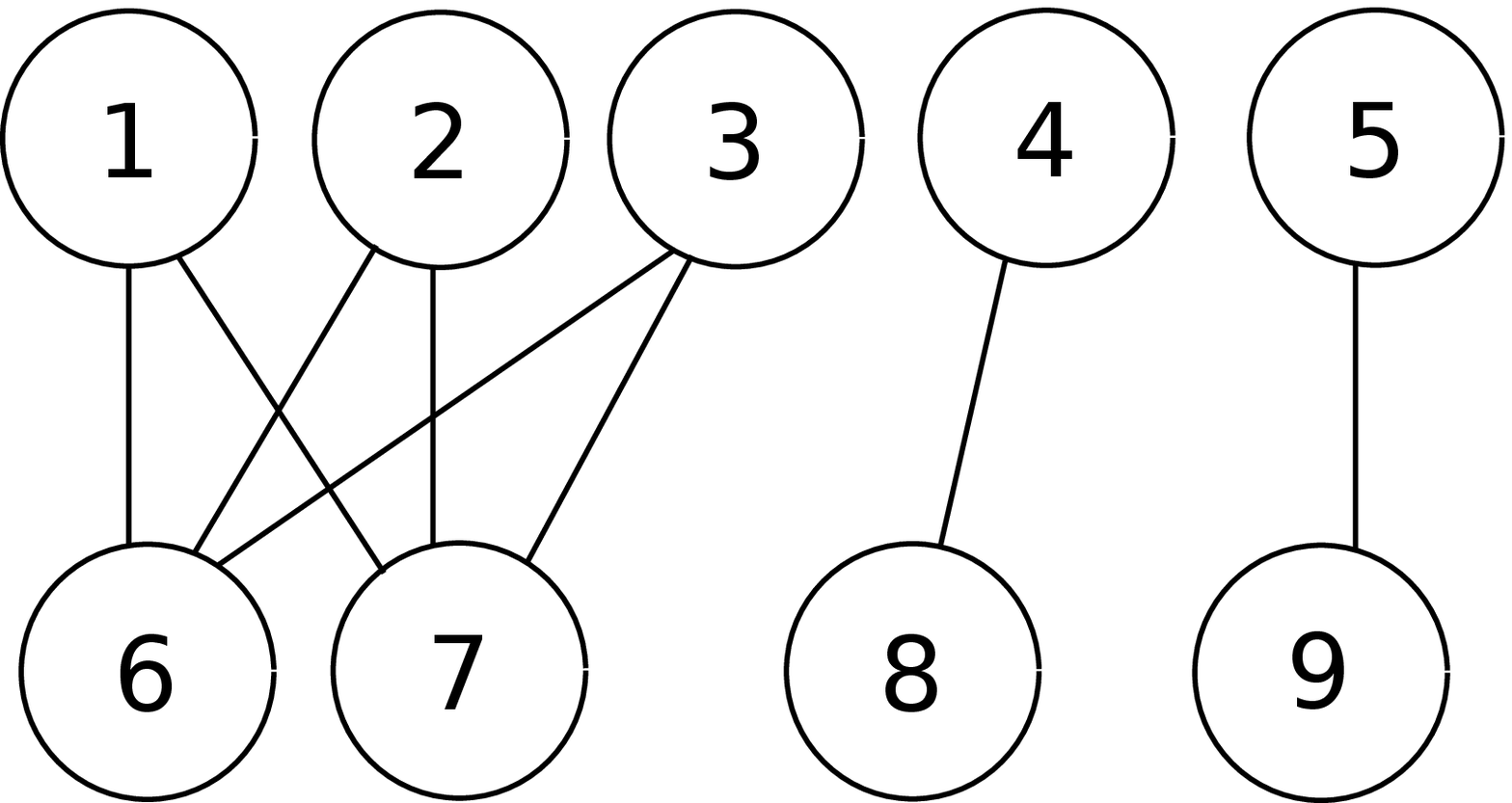}\label{fig:tras:sol_bgep}} \\
{\centering
\subfloat[\gls{mcfp} instance.]{\large
 \begin{tabular}{c!{\VRule[2pt]}cccc!{\VRule[2pt]}}
\multicolumn{1}{c}{} & $6$ & $7$ & $8$ &   \multicolumn{1}{c}{$9$}\\ \cmidrule[2pt]{2-5}
$1$ &  1 & 1 & 0 & 0 \\
$2$ &  1 & 1 & 0 & 0 \\
$3$ &  0 & 1 & 0 & 0 \\
$4$ &  0 & 0 & 1 & 1 \\
$5$ &  0 & 0 & 0 & 1 \\
\cmidrule[2pt]{2-5}
 \end{tabular}
\label{fig:tras:ins_mcfp}}\hspace*{2cm}
\subfloat[\gls{mcfp} solution.]{\large
\begin{tabular}{ccccc}
\multicolumn{1}{c}{} & $6$ &  \multicolumn{1}{c}{$7$} & $8$  & \multicolumn{1}{c}{$9$}\\  \cmidrule[2pt]{2-3}
\multicolumn{1}{c!{\VRule[2pt]}}{$1$} &  1 & \multicolumn{1}{c!{\VRule[2pt]}}{1} & 0 & 0  \\
\multicolumn{1}{c!{\VRule[2pt]}}{$2$} &  1 & \multicolumn{1}{c!{\VRule[2pt]}}{1} & 0 & 0  \\
\multicolumn{1}{c!{\VRule[2pt]}}{$3$} &  0 & \multicolumn{1}{c!{\VRule[2pt]}}{1} & 0 & 0  \\ \cmidrule[2pt]{2-4}
$4$ &  0 & \multicolumn{1}{c!{\VRule[2pt]}}{0} & \multicolumn{1}{c!{\VRule[2pt]}}{1} &  1 \\ \cmidrule[2pt]{4-5}
$5$ &  0 & 0 & \multicolumn{1}{c!{\VRule[2pt]}}{0} &  \multicolumn{1}{c!{\VRule[2pt]}}{1} \\ \cmidrule[2pt]{5-5}
 \end{tabular}
\label{fig:tras:sol_mcfp}}}
\caption{\gls{bgeps} $\leftrightarrow$ \gls{mcfp} example.} \label{fig:ex_transform}
\end{figure}

\subsection{A First Exact Algorithm for the MCFP}\label{sec:exact_algorithm}

In this section, we propose an exact iterative method for the \gls{mcfp}. In Lemma~\ref{lem:limite} we describe upper/lower bounds for the \gls{mcfp}. Next, we define a parameterized version of the \gls{bgeps} to be used in the exact algorithm.

\begin{lemma}\label{lem:limite}
Let $b^* = a^*+d^*$ be the optimal solution value of the \gls{bgep} for an instance $G$, where $a^*$ and $d^*$ are the number of edge additions and deletions, respectively, and let $M$ be the instance of the \gls{mcfp} corresponding to $G$. Then $m/(m+a^*+d^*)$ is an upper bound and $(m-d^*)/(m+a^*)$ is a lower bound for the optimal solution value of the \gls{mcfp} with input $M$, where $m=|E(G)|$.
\end{lemma}

\begin{proof}
Let $\mu(a,d) = (m-d)/(m+a)$ be the objective function of the \gls{mcfp}, where $d$ and $a$ denote, respectively, the number of ones outside cells and zeros inside cells in $M$. Consider also that $a+d=k$, for a positive constant $k$.

Taking $\mu(a,d)$ as a function $f(a)$ of $a$, we obtain that

\begin{align*}
 \mu(a,d)=\frac{m-d}{m+a} &= \frac{m-(k-a)}{m+a} = f(a).
\end{align*}

Calculating the derivative,

\begin{align*}
 \frac{df}{da} &= \frac{(m+a)-1(m-k+a)}{m^2 + 2ma + a^2} =  \frac{k}{m^2 + 2ma + a^2} > 0.
\end{align*}

Since $\frac{df}{da}> 0$ for every $a$, $f(a)$ is an increasing function. Since $k \geq a$, we have $f(k)\geq f(a)$, and thus

\begin{align*}
 \frac{m}{m+d+a} &\geq \frac{m-d}{m+a}.
\end{align*}

\noindent In other words, in the best case, the $k$ edge editing operations would correspond to $a=k$ edge additions and $d=0$ edge deletions, since as $a$ increases, $f(a)$ increases as well.

Now consider that $\mu'(a,d) = a+d$ is a feasible solution value of the \gls{bgep}. It follows that $a^*+d^* \leq a+d$ and
\begin{align*}
 \frac{m}{m+d^*+a^*} \geq \frac{m}{m+d+a} > \frac{m-d}{m+a}.
\end{align*}

Therefore, ${m}/{(m+a^*+d^*)}$ is an upper bound for the optimal solution value of the \gls{mcfp}.

Showing that $(m-d^*)/(m+a^*)$ is a lower bound is trivial since $b^* = a^*+d^*$  is a feasible solution value of the \gls{mcfp}, as shown in Lemma~\ref{lem:solution}.
\end{proof}

We now define a parameterized version of the \gls{bgeps}, the \gls{bgepsp}, which consists of finding a solution of the \gls{bgeps} with exactly $\lambda$ edge editing operations, such that the number of edge deletions is minimized. A formulation for the \gls{bgepsp} is described below:

{\allowdisplaybreaks
\begin{align}
\min &\quad \sum_{+(ij)}{(1-y_{ij})} &\\
st &\quad  y_{il} + y_{kj} + y_{kl} \leq 2 + y_{ij} &\forall i,k \in U ~and~ j,l \in V \\
&\quad \sum_{j\in V}{y_{ij}} \geq s_r & \forall i \in U \\
&\quad \sum_{i\in U}{y_{ij}} \geq s_c & \forall j \in V \\
&\quad \sum_{+(ij)}{(1-y_{ij})} + \sum_{-(ij)}{y_{ij}} = \lambda &\\
\label{cut_opt}&\quad \sum_{+(ij)}{(1-y_{ij})} \leq U_{opt}-1 &\\
&\quad x_{ij} \in \{0,1\} &\forall i \in U ~and~ j \in V.
\end{align}}

We now describe the exact iterative method for the \gls{mcfp} (Algorithm~\ref{alg:exact_cell} below). The idea is to make several calls to the \gls{bgepsp}. At each iteration, we seek for a solution with fewer deletions. Constraint \eqref{cut_opt} tells the model that the optimal value is less than $U_{opt}$, which is obtained using previous feasible solutions.

\begin{algorithm}[hbtp]
\caption{Exact iterative method for the \gls{mcfp}}
\label{alg:exact_cell}
\begin{algorithmic}[1]
\Procedure{ECM}{instance $M$}
 \State{let $G$ be the instance of the \gls{bgeps} corresponding to $M$}
 \State{$(a^*,d^*) \gets BGEPS[G]$} \Comment{Solve the \gls{bgeps} for input $G$, obtaining $a^*$ additions and $d^*$ deletions}
 \State{$\mathit{UB} \gets \frac{m}{m+a^*+d^*}$}\label{lin:lim_sup}
 \State{$\mathit{LB} \gets \frac{m-d^*}{m+a^*}$}\label{lin:lim_inf}
 \State{$U_{opt} \gets d^*$} \Comment{A bound for the number of deletions already found}
 \State{$cont \gets 0$}
 \While{$\mathit{UB} > \mathit{LB}$}
  \State{$(a,d) \gets BGEPS(a^*+d^*+cont)[G]$}\label{lin:bgeps} \Comment{Solve the \gls{bgepsp} with $\lambda=a^*+d^*+cont$ for $G$}
  \If{$U_{opt} > d$}
    \State{$U_{opt} \gets d$}
  \EndIf
  \If{$\frac{m-d}{m+a} > \mathit{LB}$}
   \State{$\mathit{LB} \gets \frac{m-d}{m+a}$} \label{lin:lim_inf2}
   \State{$(a^*,d^*) \gets (a,d)$}
  \EndIf
  \State{$\mathit{UB} \gets \frac{m}{m+a+d}$}\label{lin:lim_sup2}
  \State{$cont \gets cont+1$}
 \EndWhile
 \State{\Return{$(a^*,d^*)$}}
\EndProcedure
\end{algorithmic}
\end{algorithm}

The correctness of Algorithm~\ref{alg:exact_cell} is shown in the next result.

\begin{theorem}\label{teo:algorithm_exato}
Algorithm ECM (Algorithm \ref{alg:exact_cell}) returns an optimal solution value for the \gls{mcfp}.
\end{theorem}

\begin{proof}
Algorithm \ref{alg:exact_cell} seeks the optimal solution value iteratively through several calls to the the \gls{bgepsp}, starting with the number of edge editing operations of the \gls{bgeps}. At each iteration new bounds are calculated (variables $\mathit{LB}$ and $\mathit{UB}$). Iterations are performed until the upper bound is equal to the lower bound, meaning that from that point there is no better solution.

By Lemma~\ref{lem:limite}, lines \ref{lin:lim_sup}, \ref{lin:lim_inf}, and \ref{lin:lim_inf2} indeed  calculate the lower and the upper bounds. It remains to show that line \ref{lin:lim_sup2} actually calculates an upper bound.

The value of an optimal solution $G^*$ of the \gls{bgeps} is not necessarily optimal for the \gls{mcfp}. In this case, the optimal solution value $\mu^*$ of the \gls{mcfp} corresponds to a solution of the \gls{bgeps} with more editing operations than $G^*$. To find $\mu^*$, we use the \gls{bgepsp}. The objective function of the \gls{bgepsp} leads to a solution with the maximum number of edge additions among all solutions with exactly $\lambda$ editing operations.

In line~\ref{lin:bgeps}, a call to the \gls{bgepsp} is made, using $\lambda = a^*+d^* + cont$  editing operations, and new values $a$ and $d$ are calculated. Then line \ref{lin:lim_sup2} actually defines a new upper bound, because if there were a solution value between the new bound and the previous bound, the algorithm would have already found such a value in a previous iteration. Since the upper bound decreases along the iterations, we conclude that Algorithm \ref{alg:exact_cell} works correctly.
\end{proof}

\subsection{New linear-fractional model for the MCFP}

According to Lemma \ref{lem:solution}, the \gls{bgeps} and the \gls{mcfp} have the same feasible solution space. Thus, the constraints of the \gls{bgeps} formulation can be used in a new formulation for the \gls{mcfp}. The adaptation of the objective function is made according to the grouping efficacy described in Section~\ref{sec:of_cell}. By setting $N^{out}_{1} = \sum_{+ (ij)}{(1-y_ {ij})}$ and $N^{in}_{0} = \sum_{-(ij)}{y_{ij}}$, we propose a new formulation for the \gls{mcfp} based on linear-fractional programming:

{\allowdisplaybreaks
\begin{align*}
\max &\quad \frac{m - \sum_{+(ij)}{(1-y_{ij})} }{m + \sum_{-(ij)}{y_{ij}} }	&\\
st &\quad y_{il} + y_{kj} + y_{kl} \leq 2 + y_{ij} &\forall i,k \in U ~and~ j,l \in V \\
&\quad \sum_{j\in V}{y_{ij}} \geq s_r & \forall i \in U \\
&\quad \sum_{i\in U}{y_{ij}} \geq s_c & \forall j \in V \\
&\quad y_{ij} \in \{0,1\} &\forall i \in U ~and~ j \in V.
\end{align*}}

\subsection{Linear Formulation for the MCFP}\label{sec:linear}
In this section we propose a linearization for the linear-fractional model described above. The process consists of adding binary variables $x_{da}$ that assume value 1 if and only if the solution contains $d$ deletions and $a$ additions. Each variable $x_{da}$ is associated with a cost $c_{da} = \frac{m - d}{m + a}$  in the objective function. Let $l_c$ and $l_b$ be lower bounds for the \gls{mcfp} and \gls{bgep}, respectively. Define a set $\mathcal F = \{ (d,a) | c_{da} \geq l_c \text{~and~}  d + a \geq l_b\}$. In this case, $l_c$ can be obtained heuristicaly, whereas $l_b$ can be, for example, the value of the linear relaxation of the \gls{bgep}.

The linear formulation proposed is as follows:

\begin{align}
\label{lf1} \max \quad&\sums_{\forall ~ (d, a) \in \mathcal F}( c_{da}~ x_{da}) &\\
\label{lf2} \text{st} \quad&y_{ik} + y_{lj} + y_{lk} \leq  2 + y_{ij} & \forall~ i, l \in U \text{~and~} k, j \in V\\
\label{lf3} &\sums_{\forall ~ (d, a) \in \mathcal F} x_{da} =  1&\\
\label{lf4} &\sums_{\forall ~ (d, a) \in \mathcal F}  (d ~x_{da}) =  \sum_{-(ij)} y_{ij}&\\
\label{lf5} &\sums_{\forall ~ (d, a) \in \mathcal F}  (a ~x_{da}) =  \sum_{+(ij)} (1-y_{ij})&\\
\label{lf6} &x_{da} \in \{0,1\} &\forall~ (d, a) \in \mathcal F\\
\label{lf7} &y_{ij} \in \{0,1\} &\forall~ i \in U\text{~and~}j \in V
\end{align}

The objective function \eqref{lf1} computes the maximum cost $c_{da} = \frac{m - d}{m + a}$. Constraints \eqref{lf2} eliminate induced subgraphs isomorphic to $P_4$. Constraint \eqref{lf3} imposes that exactly one $x_{da}$ variable should assume value 1. Constraints \eqref{lf4} and \eqref{lf5} state that the number of deletions and additions must be $d$ and $a$, respectively. Constraints \eqref{lf6} and \eqref{lf7} define the domain of the variables.

\subsection{MCFP Instances}\label{sec:ins_mcfp}

The \gls{mcfp} has been explored in the literature for many years, and several works have proposed instances for this problem. In this paper, we used 35 instances available in \citet{Goncalves2004}, which have been used in many other papers. In Table~\ref{tab:instances_mcfp}, we present the instances with its dimensions. To the best of our knowledge, this work is the first one that finds the optimal solutions of all such instances.

\begin{table}[hbtp]
\centering
\footnotesize
\begin{tabular}{||l|c||l|c||}
\toprule
Instance            & Dimension     & Instance                  & Dimension	\\ \midrule
King1982            &$05\times07$   &Kumar1986                  &$20\times23$	\\
Waghodekar1984      &$05\times07$   &Carrie1973b                &$20\times35$	\\
Seifoddini1989      &$05\times18$   &Boe1991                    &$20\times35$	\\
Kusiak1992          &$06\times08$   &Chandrasekharan1989\_1     &$24\times40$	\\
Kusiak1987          &$07\times11$   &Chandrasekharan1989\_2     &$24\times40$	\\
Boctor1991          &$07\times11$   &Chandrasekharan1989\_3-4   &$24\times40$	\\
Seifoddini1986		&$08\times12$   &Chandrasekharan1989\_5     &$24\times40$	\\
Chandrasekaran1986a	&$08\times20$   &Chandrasekharan1989\_6     &$24\times40$	\\
Chandrasekaran1986b	&$08\times20$   &Chandrasekharan1989\_7     &$24\times40$	\\
Mosier1985a         &$10\times10$   &McCormick1972b             &$27\times27$	\\
Chan1982            &$10\times15$   &Carrie1973c                &$28\times46$	\\
Askin1987           &$14\times24$   &Kumar1987                  &$30\times51$	\\
Stanfel1985         &$14\times24$   &Stanfel1985\_1             &$30\times50$	\\
McCormick1972a      &$16\times24$   &Stanfel1985\_2             &$30\times50$	\\
Srinivasan1990      &$16\times30$   &King1982                   &$30\times90$	\\
King1980            &$16\times43$   &McCormick1972c             &$37\times53$	\\
Carrie1973a         &$18\times24$   &Chandrasekharan1987        &$~~40\times100$\\
Mosier1985b         &$20\times20$   & & \\ \bottomrule
\end{tabular}
\caption{Instances of the \gls{mcfp}.}\label{tab:instances_mcfp}
\end{table}

\section{Computational Experiments}\label{sec:results}

In this section we evaluate and compare the algorithms proposed in this work. We use the mixed linear optimization software CPLEX \citep{Cplex}, which is responsible for managing the Branch-and-Cut method, including:
\begin{itemize}
  \item choice of variables for the branch;
  \item execution of the separation algorithm;
  \item addition of cuts generated by the preprocessing procedure.
\end{itemize}

All the algorithms have been run on an Intel Core i7-2600 3.40 GHz machine with 32 GB of RAM and Arch Linux 3.3.4 operating system. The separation algorithm in Section~\ref{sec:sep}, the preprocessing procedure in Section~\ref{sec:pre}, and the exact iterative method for the \gls{mcfp} in Section~\ref{sec:exact_algorithm} have been implemented in C++. All the instances and solutions are available at \url{http://www.ic.uff.br/\~fabio/instances.pdf}.

\subsection{Experimental Results for the BGEP}

The proposed Branch-and-Cut algorithm for the \gls{bgep} was applied to 30 randomly generated instances, as explained in Section~\ref{sec:instances_bgep}, with $p \in \{0.2,0.4,0.6,0.8\}$, $m \in [10,21]$, and $n \in [11,22]$.

We first compare the default separation algorithm incorporated in CPLEX with the separation algorithm described in Section~\ref{sec:sep}. To evaluate these two running scenarios, we use the {\em geometric mean}. The geometric mean of a data set $\{t_1, t_2, \dotsc, t_n \} $ is defined as:
\begin{align}\label{eq:geometric}
 G = \sqrt[n]{t_1 t_2 \dotsm t_n}.
\end{align}

Each scenario is applied to all the $30$ random instances. Table~\ref{tab:comparison} shows the results. From left to right, the columns of the table show: separation algorithm, total running time over all the $30$ instances, geometric mean of the $30$ computational times, and number of times each scenario achieves the best running time. Note that the separation algorithm presented in Section~\ref{sec:sep} clearly influences the convergence speed.

\begin{table}[hbtp]\centering
\begin{footnotesize}
\begin{tabular}{|c|r|c|c|}
 \hline
{\bf Separation algorithm} & {\bf Total time} (s) & {\bf Geometric mean} & {\bf Number of best results}\\ \hline
Dynamic Programming & 77235.28  & 15.92 & 16 \\ \hline
CPLEX Default       & 131068.72 & 16.33 & 14 \\ \hline
\end{tabular}
\end{footnotesize}
\caption{Comparison of separation algorithms for 30 random instances.} \label{tab:comparison}
\end{table}

We now analyze the impact of the preprocessing procedure described in Section~\ref{sec:pre}. Results are presented in Table~\ref{tab:pre}. Each row in the table shows results for instances generated with the same value of $p$. We remark that, in the $\mathbb G(m,n,p)$  model, the probability $p$ is the expected edge density of a random bipartite instance (i.e., the expected number of edges is $mnp$). From left to right, the columns show: instance density, average percentage of fixed variables, and average percentage of generated cuts. (The maximum number of cuts is $mn(n-1)/2 + nm(m-1)/2$, which is achieved by an edgeless bipartite graph.)

\begin{table}[hbtp]\centering
\begin{footnotesize}
\begin{tabular}{|c|c|c|}
\hline
{\bf Instance density} & {\bf Average percentage of fixed variables} &{\bf Average percentage of generated cuts}\\ \hline
0.2 & 19\% & 51\% \\ \hline
0.4 & 1\%  & 11\% \\ \hline
0.6 & 0\%  & 0\%  \\ \hline
0.8 & 0\%  & 0\%  \\ \hline
\end{tabular}
\end{footnotesize}
\caption{Impact of the preprocessing procedure.} \label{tab:pre}
\end{table}

As the instances get denser, the effectiveness of the preprocessing procedure decreases. This is due to the fact that, in a sparse instance, the probability that two vertices are at a distance at least $4$ (see Theorem~\ref{teo:distance}) is greater than in a dense instance.

\subsection{Experimental results for the MCFP}

Table~\ref{tab:exato_sig} shows the results of the application of the exact iterative method for the \gls{mcfp} presented in Section~\ref{sec:exact_algorithm} (with singleton constraints) on the instances shown in Table~\ref{tab:instances_mcfp}. From left to right, the columns show: instance name; best solution value found in the literature \citep{Goncalves2004,Wu2008,Wu2010}; optimum value found by our iterative method (using the grouping efficacy as the objective function); the sum $N^{out}_{1} + N^{in}_{0}$; and the number of edge editing operations obtained by the application of our Branch-and-Cut method presented in Section~\ref{sec:branch} on the corresponding \gls{bgeps} instance (as explained in Lemma~\ref{lem:solution}). We remark that instances marked with * have been erroneously encoded in some works.

Our method finds the optimum for 27 of 35 instances. These optima were previously unknown.
For instances King1980 and Kumar1987, it finds solutions better than the existing solutions in the literature. A point to note is that the values in third and fourth columns are very close; the difference is only $1.81$ on average. In particular, for 17 of the 27 optima in the fourth column, these values coincide. This shows that optimal solutions for the \gls{bgeps} correspond to quite good solutions for the \gls{mcfp}.

Table~\ref{tab:exato_sem_sig} is similar to Table~\ref{tab:exato_sig} and shows the results for the exact iterative method for the \gls{mcfp} without singleton constraints. Values in the second column are taken from \citep{Pailla2010,Wu2009,Elbenani2012}. As in the previous table, instances marked with * have been erroneously encoded in some works.

Again, our method finds the optimum for 27 of 35 instances. For three instances (King1980, Kumar1987, and Stanfel1985\_1), it finds solutions better than the existing ones in the literature. The difference between values in the third and fourth columns is only $1.13$ on average (the difference is zero in $16$ cases).

\begin{table}[htbp]
\centering
\begin{scriptsize}
\begin{tabular}{||l||c|c|c|c||}
\toprule
Instance & Literature & Optimum & $N^{out}_{1} + N^{in}_{0}$ & \gls{bgeps} Optimum \\
\midrule
King1982                 & 73.68      & 73.68          & \textbf{5}     & \textbf{5} \\
Waghodekar1984           & 62.50      & 62.50          & \textbf{9}     & \textbf{9} \\
Seiffodini1989           & 79.59      & 79.59          & \textbf{10}    & \textbf{10}\\
Kusiak1992               & 76.92      & 76.92          & \textbf{5}     & \textbf{5} \\
Kusiak1987               & 53.13      & 53.12          & \textbf{15}    & \textbf{15}\\
Boctor1991               & 70.37      & 70.37          & \textbf{7}     & \textbf{7} \\
Seiffodini1986           & 68.30      & 68.29          & \textbf{13}    & \textbf{13}\\
Chandrasekaran1986a      & 85.25      & 85.24          & \textbf{9}     & \textbf{9} \\
Chandrasekaran1986b	     & 58.72      & 58.71          & 45             & 43         \\
Mosier1985a              & 70.59      & 70.58          & \textbf{10}    & \textbf{10}\\
Chan1982                 & 92.00      & 92.00          & \textbf{4}     & \textbf{4} \\
Askin1987                & 69.86      & 69.86          & 22             & 21         \\
Stanfel1985              & 69.33      & 69.33          & 23             & 22         \\
McCornick1972a           & 51.96      & 51.96          & 49             & 46         \\
Srinivasan1990           & 67.83      & 67.83          & \textbf{46}    & \textbf{46}\\
King1980                 & 55.90      & \textbf{56.52} & 70             & 68         \\
Carrie1973a              & 54.46      & 54.46          & 51             & 47         \\
Mosier1985b              & 42.96      &                &                &            \\
Kumar1986                & 49.65      & 49.64          & 71             & 64         \\
Carrie1973b              & 76.54      & 76.54          & \textbf{37}    & \textbf{37}\\
Boe1991                  & 58.15      & 58.15          & 77             & 73         \\
Chandrasekharan1989\_1   & 100.00     & 100.00         & \textbf{0}     & \textbf{0} \\
Chandrasekharan1989\_2   & 85.11      & 85.10          & \textbf{21}    & \textbf{21}\\
Chandrasekharan1989\_3-4 & 73.51      & 73.50          & \textbf{39}    & \textbf{39}\\
Chandrasekharan1989\_5   & 51.97      & 51.97          & 73             & 70         \\
Chandrasekharan1989\_6   & 47.37      &                &                &            \\
Chandrasekharan1989\_7   & 44.87      &                &                &            \\
McCormick1972b           & 54.27      &                &                &            \\
Carrie1973c              & 46.06      &                &                &            \\
Kumar1987$^*$            & 58.58      & \textbf{58.94} & 62             & 60         \\
Stanfel1985\_1$^*$       & 59.66      & 59.65          & \textbf{71}    & \textbf{71}\\
Stanfel1985\_2           & 50.51      &                &                &            \\
King1982                 & 42.64      &                &                &            \\
McCormick1972c           & 59.85      &                &                &            \\
Chandrasekharan1987      & 84.03      & 84.03          & \textbf{73}    & \textbf{73}\\
\bottomrule
\end{tabular}
\end{scriptsize}
\caption{Results of the exact iterative method for the MCFP with singleton constraints $2 \times 2$.}
\label{tab:exato_sig}
\end{table}

\begin{table}[htbp]
\centering
\begin{scriptsize}
\begin{tabular}{||l||c|c|c|c||}
\toprule
Instance & Literature & Optimum & $N^{out}_{1} + N^{in}_{0}$ & \gls{bgeps} Optimum \\
\midrule
King1982$^*$             & 75    & 75             & \textbf{4}  & \textbf{4} \\
Waghodekar1984           & 69.57 & 69.56          & \textbf{7}  & \textbf{7} \\
Seiffodini1989           & 80.85 & 80.85          & \textbf{9}  & \textbf{9} \\
Kusiak1992               & 79.17 & 79.16          & \textbf{5}  & \textbf{5} \\
Kusiak1987               & 60.87 & 60.86          & \textbf{9}  & \textbf{9} \\
Boctor1991               & 70.83 & 70.83          & \textbf{7}  & \textbf{7} \\
Seiffodini1986           & 69.44 & 69.44          & \textbf{11} & \textbf{11} \\
Chandrasekaran1986a      & 85.25 & 85.24          & \textbf{9}  & \textbf{9} \\
Chandrasekaran1986b      & 58.72 & 58.71          & 45          & 43 \\
Mosier1985a              & 75    & 75             & \textbf{7}  & \textbf{7} \\
Chan1982                 & 92    & 92             & \textbf{4}  & \textbf{4} \\
Askin1987                & 74.24 & 74.24          & \textbf{17} & \textbf{17} \\
Stanfel1985              & 72.86 & 72.85	      & 19          & 18 \\
McCornick1972a           & 53.33 & 53.33 	      & \textbf{42} & \textbf{42} \\
Srinivasan1990           & 69.92 & 69.92 	      & 40          & 39 \\
King1980                 & 57.96 & \textbf{58.04} & 60          & 59 \\
Carrie1973a              & 57.73 & 57.73          & 41          & 40 \\
Mosier1985b              & 43.97 &                &             &    \\
Kumar1986                & 50.81 & 50.80          & 61          & 60 \\
Carrie1973b              & 79.38 & 79.37          & 33          & 32 \\
Boe1991                  & 58.79 & 58.79          & 75          & 70 \\
Chandrasekharan1989\_1   & 100	 & 100            & \textbf{0}  & \textbf{0} \\
Chandrasekharan1989\_2   & 85.11 & 85.10          & \textbf{21} & \textbf{21} \\
Chandrasekharan1989\_3-4 & 73.51 & 73.50          & 40          & 39 \\
Chandrasekharan1989\_5   & 53.29 & 53.28          & 71          & 64   \\
Chandrasekharan1989\_6   & 48.95 &                &             &    \\
Chandrasekharan1989\_7   & 46.58 &                &             &    \\
McCormick1972b           & 54.82 &                &             &    \\
Carrie1973c              & 47.68 &                &             &    \\
Kumar1987$^*$            & 62.86 & \textbf{63.04} & \textbf{51} & \textbf{51} \\
Stanfel1985\_1$^*$       & 59.66 & \textbf{59.77} & 70          & 67 \\
Stanfel1985\_2           & 50.83 &                &             &    \\
King1982                 & 47.93 &                &             &    \\
McCormick1972c           & 61.16 &                &             &    \\
Chandrasekharan1987      & 84.03 & 84.03          & \textbf{73} & \textbf{73} \\
\bottomrule
\end{tabular}
\end{scriptsize}
\caption{Results of the exact iterative method for the MCFP without singleton constraints.}\label{tab:exato_sem_sig}
\end{table}

In Table~\ref{tab:comp_mcfp} we compare the running times obtained by applying the two exact approaches for the \gls{mcfp} proposed in this work (disregarding singleton constraints) on the instances of Table~\ref{tab:instances_mcfp}. The linear model described in Section~\ref{sec:linear} achieves  running times faster than those by the exact iterative method in Section~\ref{sec:exact_algorithm}, with few exceptions. (The addition of singleton constraints produces similar results.)

\begin{table}[hbtp]
\centering
\footnotesize
\begin{tabular}{|c|c|c|}
\hline
{\bf Instance} 		     & {\bf Linear Model} & {\bf Iterative Model} \\
\hline
\hline
King1982                 & 0.01     & 0.16      \\ \hline
Waghodekar1984           & 0.01     & 0.07      \\ \hline
Seifoddini1989           & 0.03     & 0.09      \\ \hline
Kusiak1992               & 0.01     & 0.02      \\ \hline
Boctor1991               & 0.01     & 0.14      \\ \hline
Kusiak1987               & 0.06     & 0.29      \\ \hline
Seifoddini1986		     & 0.03     & 0.18      \\ \hline
Chandrasekharan1986a     & 0.04     & 2.06      \\ \hline
Chandrasekharan1986b     & 4.94     & 81.46     \\ \hline
Mosier1985a              & 0.01     & 0.03      \\ \hline
Chan1982                 & 0.02     & 0.01      \\ \hline
Askin1987                & 0.09     & 0.49      \\ \hline
Stanfel1985              & 0.11     & 0.49      \\ \hline
McCormick1972            & 144.91   & 600.98    \\ \hline
Srinivasan1990           & 0.54     & 7.24      \\ \hline
King1980                 & 125.62   & 1156.23   \\ \hline
Carrie1973               & 42.32    & 87.13     \\ \hline
Mosier1985b              & --       & --        \\ \hline 
Kumar1986                &1771.99   & 23928.70  \\ \hline
Boe1991                  & 305.48   & 2145.24   \\ \hline
Carrie1973               & 14.55    & 1.78      \\ \hline
Chandrasekharan1989\_1   & 0.15     & 0.02      \\ \hline
Chandrasekharan1989\_2   & 0.44     & 10.08     \\ \hline
Chandrasekharan1989\_3-4 & 0.78     & 17.46     \\ \hline
Chandrasekharan1989\_5   &48743.90  & 371233.00 \\ \hline
Chandrasekharan1989\_6   & --       & --        \\ \hline 
Chandrasekharan1989\_7   & --       & --        \\ \hline 
McCormick1972          	 & --       & --        \\ \hline 
Carrie1973               & --       & --        \\ \hline 
Kumar1987                & 41.53    & 183.71    \\ \hline
Stanfel1985\_1           &2622.06   & 13807.50  \\ \hline
Stanfel1985\_2           & --       & --        \\ \hline 
King1982                 & --       & --        \\ \hline 
McCormick1972            & --       & --        \\ \hline 
Chandrasekharan1987      & 18.22    & 325.53    \\ \bottomrule
\end{tabular}
\caption{Comparison between running times (in seconds) of the exact methods for the MCFP.}\label{tab:comp_mcfp}
\end{table}

\section{Conclusions}\label{sec:conclusions}

This work has investigated the close relationships between the \gls{bgep} and the \gls{mcfp}. This opens new possibilities of research, in the sense that each every contribution to one of the problems may be applied to the other.

We have proposed a new Branch-and-Cut method for the \gls{bgep} based on a dynamic programming separation algorithm. Our method has been applied to $30$ randomly generated instances with edge densities ranging from $0.2$ to $0.8$ and sizes from $110$ to $462$ vertices. Experimental results show that the proposed method outperforms the CPLEX standard separation algorithm.

We have also described a new preprocessing procedure for the \gls{bgep} based on theoretical developments related to vertex distances in the input graph (Theorem~\ref{teo:distance}). The procedure is effective to fix variables and generate cuts for low density random instances.

The similarity between the \gls{bgep} and \gls{mcfp} has been explored. We have shown that these problems have the same feasible solution space. This fact allows the use of mathematical formulations for the \gls{bgep} to solve the \gls{mcfp}. It is worth remarking that the combinatorial structure of the \gls{bgep} can be directly applied to the solution of the \gls{mcfp}. An example is the use of constraints (11) and (18) in two formulations for the \gls{mcfp} described in this work; such constraints are used to eliminate induced subgraphs isomorphic to $P_4$, which are forbidden for bicluster graphs.

We have shown that good solutions of the \gls{bgep} correspond to good solutions of the \gls{mcfp}, i.e., decreasing the number of editing operations corresponds to obtaining a matrix  permutation with fewer 1's outside and 0's inside diagonal blocks.

Our contributions to the \gls{mcfp} include:
\begin{itemize}
 \item a new exact iterative method based on several calls to a parameterized version of the \gls{mcfp};
 \item a new linear-fractional formulation and its linearization;
 \item a experimental study of the impact of using a linear objective function for the \gls{mcfp};
 \item a separated analysis of problems with or without singleton constraints;
 \item exact resolution of most instances of the \gls{mcfp} found in literature, some crated almost $40$ years ago.
\end{itemize}

Ongoing work includes the study of other variants of the \gls{mcfp} and their correspondence with adapted versions of the \gls{bgep}.

\bibliographystyle{model1-num-names}
\bibliography{refs}

\end{document}